\documentclass[11pt]{amsart}
\usepackage[margin=1.2in]{geometry}
\usepackage{amsmath}
\usepackage{color}
\usepackage{graphicx}
\usepackage[pdfstartview=FitH]{hyperref}
\usepackage{tikz-cd}
\usetikzlibrary{patterns}
\usetikzlibrary{petri}
\usepackage{float}

\newcommand{\N}{{\mathbb{Z}} _{> 0} }
\newcommand{\Zp} {\Z _ {\ge 0} }
\newcommand{\Z}{\mathbb{Z}}
\newcommand{\C}{\mathbb{C}}

\providecommand{\U}[1]{\protect\rule{.1in}{.1in}}

\newcommand{\HVir}{\mathcal H}

\newtheorem{theorem}{Theorem}[section]
\newtheorem*{theorem*}{Theorem}
\newtheorem{corollary}[theorem]{Corollary}
\newtheorem{lemma}[theorem]{Lemma}

\newtheorem{remark}[theorem]{Remark}

\newtheorem{definition}[theorem]{Definition}
\newtheorem{proposition}[theorem]{Proposition}

\DeclareMathOperator{\Ker}{Ker}
\DeclareMathOperator{\Img}{Im}
\newcommand{\bea}{\begin{eqnarray}}
\newcommand{\eea}{\end{eqnarray}}

\begin{document}
\title[{Self-dual and logarithmic representations } ]{Self-dual and logarithmic representations of the twisted Heisenberg--Virasoro algebra at level zero}
\author{Dra\v zen Adamovi\' c }
\curraddr{Department of Mathematics, Faculty of Science, University of Zagreb, Bijeni\v cka 30, 10 000
Zagreb, Croatia}
\email{adamovic@math.hr}
\author{Gordan Radobolja }
\curraddr{Faculty of Science, University of Split, Ru\dj era Bo\v{s}kovi\'{c}a 33,
21 000 Split, Croatia }
\email{gordan@pmfst.hr}
\date{}
\maketitle

\begin{abstract}
This paper is a continuation of \cite{AR1}. We present certain new applications and generalizations of the free field realization of the twisted Heisenberg--Virasoro algebra $\HVir$ at level zero.
  We find explicit formulas for singular vectors in certain Verma modules. A free field realization of self-dual modules for $\HVir$ is presented by combining a bosonic construction of Whittaker modules from \cite{ALZ} with a construction of logarithmic modules for vertex algebras. As an application, we prove that there exists a non-split self-extension of irreducible self-dual module which is a logarithmic module of rank two. 
  We construct a large family of logarithmic modules containing  different types of highest weight modules as subquotients. We believe that these logarithmic modules are related with projective covers of irreducible modules in a suitable category of $\HVir$--modules.
\end{abstract}

\baselineskip=14pt
\newenvironment{demo}[1]{\vskip-\lastskip\medskip\noindent{\em#1.}\enspace
}{\qed\par\medskip}

 \section{Introduction}
The twisted Heisenberg--Virasoro Lie algebra $\HVir$ is an important example of a Lie algebra whose associated vertex algebra has many applications in the representation theory and conformal field theory. If the level of the corresponding Heisenberg vertex subalgebra is non-zero, the Heisenberg--Virasoro vertex algebra is isomorphic to the tensor product of Heisenberg vertex algebra and the (universal or simple) Virasoro vertex algebra (cf.\ \cite{ACKP},\cite{CW}). The study of the twisted Heisenberg--Virasoro algebra at level zero was initiated by Y.\ Billig \cite{Billig} motivated by applications to the toroidal Lie algebras. New results on the representation theory were obtained in recent papers \cite{AR1},  \cite{JJ},  \cite{R-2013}.
 
 In this paper we continue our study of free field realization of the twisted Heisenberg--Virasoro algebra from \cite{AR1}. Our main motivation is to present free field realization of self-dual modules and certain Verma modules which we were unable to construct using methods from \cite{AR1}.

 Let $\HVir$ be the twisted Heisenberg--Virasoro algebra at level zero. Let $V^{\HVir}(h,h_{I})$ (resp.\ $L^{\HVir}(h,h_{I})$) denote the Verma module (resp.\ the irreducible highest weight module) with highest weight $(c_{L},c_{I},c_{L,I},h,h_{I})$ such that $c_I=0$, $c_{L,I}\neq 0$ is fixed but arbitrary, and with highest weight vector $v_{h, h_I}$. Then $V^{\HVir}(h,h_{I})$ is reducible if and only if $h_I /c_{L,I} -1 \in {\Z} \setminus \{ 0\}$ (cf.\ \cite{Billig}).
 
A free field realization of irreducible highest weight modules for the twisted Heisenberg--Virasoro algebra at level zero was presented by the authors in \cite{AR1}. Using free field realization we calculated the fusion rules for non-generic irreducible modules i.e., for those irreducible modules which are not isomorphic to a Verma $\HVir$--module. In our approach, the screening operator $Q$ introduced in \cite[Section 2]{AR1} has played an important role. In particular, the singular vector in $V^{\HVir}(h,h_I)$ in the case 
 $$ h_I / c_{L,I} -1 = p \in \Z_{> 0} $$
 is expressed as $Q v_{h + p , h_I } = S_p (c) v_{h,h_I}$ where $S_p(c)$ is $p^{th}$--Schur polynomial in variables $(c(-1), c(-2), \cdots)$ defined by  (\ref{schur-def})  where  $c = - I / c_{L,I}$: $$S_p(c)=S_p(c(-1),\ldots,c(-p))=S_p\left(\frac{-1}{c_{L,I}}I(-1),\ldots,\frac{-1}{c_{L,I}}I(-p)\right)$$ (see (\ref{Heis}) and \cite[Theorem 4.3]{AR1}).
 
 In the present paper we found a similar approach for the singular vector in the case
\bea h_I / c_{L,I} -1 = -p \quad ( p \in \Z_{> 0} ).\label{case-drugi} \eea
The formula is 
\bea
&&   \sum_{i = 1}^{p}   \left(   L(-i) S_{p-i} (- c) \right) v_{h,h_I}  +  \left( h  +   \frac{c_L- 2}{24} (p-1)  \right ) S_p (-c)   v_{h,h_I}   \nonumber \\   &&  -  \frac{c_L- 26}{24} \left(  \sum_{i = 1}  ^{p} (i-1) c(-i )  S_{p-i} (-c) \right)  v_{h,h_I}. \label{singular-drugi}
 \eea

Verma modules $V^{\HVir}(h, h_I)$ in the case (\ref{case-drugi}) are explicitly constructed in Section \ref{verma} by using certain deformation studied previously in \cite{AM-Selecta}. Singular vector (\ref{singular-drugi}) admits a nice interpretation as an element $e^{ \frac{p-1}{2} d^2 + \frac{1-r}{2} c + c}$ of the group algebra in the explicit lattice realization (cf.\ Theorem \ref{verma-realization}).
 
It turned out that the methods from \cite{AR1} do not provide a free field realization of self-dual modules $L^{\HVir} (h, h_I)$ such that $h \ne \frac{c_L-2}{24}$,  $h_I = c_{L,I}$. On the other hand, the same paper contains certain results on vanishing of the fusion rules in the category of modules which contains self-dual module $L^{\HVir} (h, h_I)$ as above (see \cite[Theorem 5.4(3), Remark 6.5]{AR1}). In order to understand these fusion rules properties from the vertex--algebraic point of view, one needs to find a bosonic realization of self-dual modules. 
We shall see that the realization includes both the concepts of Whittaker and logarithmic modules for vertex algebras. We shall prove that the Whittaker module $\Pi_{\lambda}$ for the vertex algebra $\Pi(0)$ introduced in \cite{ALZ}, after certain logarithmic deformation (cf.\ Theorem \ref{deformation}) becomes a highly reducible $\HVir$--module $\widetilde \Pi_{\lambda}$ which contains $L^{\HVir} (h, h_I)$ as a submodule. Since $\widetilde \Pi_{\lambda}$ is not a module for the Heisenberg vertex algebra $M(1)$, it is clear that such module could not have appeared in the fusion rules analysis made in \cite{AR1}.

\vskip 5mm

In Section \ref{review-ar} we recall from \cite{AR1} a free field realization of $\HVir $, the definition of vertex algebra $\Pi(0)$ and its modules $\Pi(p,r)$. We study an extension of the Heisenberg--Virasoro vertex algebra $\overline{\Pi(0)} \subset \Pi(0)$ and present a structure of $\overline{\Pi(0)}$--modules $\Pi(p,r)$ in Section \ref{extension}. By using certain relation in $\Pi(0)$--modules we recover formula \ref{singular-drugi} in Section \ref{singv}. Then we consider a deformed action of $\HVir$ on these modules  and obtain a family of modules $\widetilde{ \Pi(p,r) }$ [cf.\ Theorem  \ref{deformation}, 
%pg.\ \pageref{deformation}; 
Theorem \ref{p-neg}]
% pg.\ \pageref{p-neg}] 
with the following properties:
\begin{itemize}
\item $\widetilde{\Pi(p,r)}$ is a logarithmic  $\HVir$--module with the following action of the element 
$\widetilde{L(0)}$ of the Virasoro algebra:
\begin{itemize}
	\item ${\C}[\widetilde{L(0)} ] v$ is finite--dimensional for every $v \in \widetilde{\Pi(p,r)}$, $p \ge 1$,
	\item ${\C}[\widetilde{L(0)} ] v$ is infinite--dimensional for every  $v \in \widetilde{\Pi(p,r)}$, $p \le 0$.
 \end{itemize}
 \item For $p \ge 1$,  $\widetilde{ \Pi(p,r)} $ admits a filtration 
 $$\widetilde{ \Pi(p,r) }\cong \bigcup_{m \ge 0} Z  ^{(m)} , $$ such that $ Z  ^{(m )} $ is a logarithmic  $\HVir$--module of rank $m+1$ and 
  $ Z  ^{(m)} /  Z  ^{(m-1)}$ is a weight $\HVir$--module. [cf.\ Theorem \ref{filtracija-2}]
%  , pg.\ \pageref{filtracija-2}]
 \end{itemize}

In Section \ref{whittaker} we consider a deformed action on Whittaker $\Pi(0)$--modules and obtain a realization of self-dual $\HVir$--modules $L^{\HVir}(h,c_{L,I})$ , $h \ne \frac{c_L-2}{24}$ [cf.\ Theorem  \ref{self-dual-real}]
% pg.\ \pageref{self-dual-real}].

Finally, we find new applications of our results on the vertex algebra associated to the $W(2,2)$--algebra. We present a non-local bosonic formula (\ref{formula-screening-2}) for the screening operator introduced by the authors in \cite{AR2}.  We hope that our new expression for screening operator can be applied to construction of a quantum group which would play the role of Kazhdan--Lusztig dual of the vertex algebra $W(2,2)$. 

We also discuss a connection of our approach with a  realization of the BMS$_3$--algebra obtained in \cite{BJMN} in the case $c_L = 26$.

\vskip 5mm
{\bf Acknowledgements.}  The authors are partially supported by the Croatian Science Foundation under the project 2634
and by the QuantiXLie Centre of Excellence, a project cofinanced by the Croatian Government and European Union through the European Regional Development Fund - the Competitiveness and Cohesion Operational Programme (Grant KK.01.1.1.01.0004). We thank the referee for his/her  valuable comments.

\section{Results from \cite{AR1}}
\label{review-ar}
In this section we recall from \cite{AR1} the free field realization of the twisted Heisenberg--Virasoro algebra, the construction of vertex algebra $\Pi(0)$ and their modules $\Pi(p,r)$. The definition of the lattice is slightly changed, but the action of the generators of the Heisenberg--Virasoro algebra is the same as in \cite{AR1}. Main results of Section 2 of \cite{AR1} are stated in Proposition \ref{kratka-verzija}. 

We also present new, explicit formulas for sub--singular vectors introduced in \cite[Proposition 2.7]{AR1}.

Recall that the twisted Heisenberg--Virasoro algebra is an infinite--dimensional complex Lie algebra $\HVir$ with basis
$$
\{L(n),I(n):n\in\mathbb{Z}\}\cup\{C_{L},C_{LI},C_{I}\}
$$
and commutation relation:
\bea
&& \left[  L(n),L(m)\right]  =(n-m)L(n+m)+\delta_{n,-m}\frac{n^{3}-n}{12}%
C_{L},\\
&& \left[  L(n),I(m)\right]  =-mI(n+m)-\delta_{n,-m}( n^{2}+n)
C_{LI},\\
 && \left[  I(n),I(m)\right]  =n\delta_{n,-m}C_{I},\\ && \left[  {\HVir} ,C_{L}\right]  =\left[  {\HVir},C_{LI}\right]  =\left[
\mathcal{H},C_{I}\right]  =0.
\eea

Let $V^{\HVir}(c_{L},c_{I},c_{L,I},h,h_{I})$ denote the Verma module with highest
weight $(c_{L},c_{I},c_{L,I},h,h_{I})$, and  $L^{\HVir}(c_{L},c_{I},c_{L,I},h,h_{I})$ its irreducible quotient (cf.\ \cite{Billig}). In this paper we consider the case $c_I=0$. For simplicity we shall denote the Verma module $V^{\HVir}(c_{L}, 0,c_{L,I},h,h_{I})$ with $V^{\HVir} (h,h_{I})$ and its irreducible quotient with $L^{\HVir}( h,h_{I})$.

\vskip 5mm

%Define the following hyperbolic lattice $L= {\mathbb{Z}} {c} +
%{\mathbb{Z}} {d}$ such that
%$$
%\langle c ,  c \rangle=  \langle d ,  d \rangle = 0, \ \langle
%c , d \rangle= 2.
%$$

Define the following hyperbolic lattice $Hyp= {\mathbb{Z}} x  +
{\mathbb{Z}}   y  $ such that
$$
\langle x ,  x \rangle=  \langle  y  ,  y   \rangle = 0, \ \langle
 x , y    \rangle= 1.
$$

Let ${\mathfrak h} = {\C} \otimes_{\Z}  Hyp$ and extend the form $\langle \cdot, \cdot \rangle $ on ${\mathfrak h}$.
We can consider ${\mathfrak h}$ as an abelian Lie algebra. Let $\widehat{\mathfrak h}= {\C}[t,t^{-1}]  \otimes {\mathfrak h} \oplus {\C} K$ be the affinization of ${\mathfrak h}$. Let $\gamma \in {\mathfrak h}$ and consider $\widehat{\mathfrak h}$--module
$$ M(1, \gamma) := U(\widehat{\mathfrak h}) \otimes _{ U ( {\C} [t] \otimes {\mathfrak h} \oplus {\C} K ) } {\C}  $$
where $ t {\C}   [t] \otimes {\mathfrak h} $ acts trivially on ${\C}$, $ {\mathfrak h}$ acts as $ \langle \delta , \gamma \rangle $ for $\delta \in {\mathfrak h} $ and $K$ acts as $1$.
We shall denote the highest weight vector in $M(1, \gamma)$ by $e ^{\gamma}$.

We shall write $M(1)$ for $M(1,0)$.  For $h \in {\mathfrak h}$ and $n \in {\Z}$ we write $h(n)$ for $t ^n \otimes h$. Set $h(z) = \sum_{n \in {\Z} } h(n)  z ^{-n-1}$. Then $M(1)$ is a vertex algebra which is generated by the fields $h(z)$, $h \in {\mathfrak h}$ (cf.\ \cite{LL}. Moreover, $M(1, \gamma)$ for $\gamma \in {\mathfrak h}$, are irreducible $M(1)$--modules.

Let $V_{Hyp } = M(1) \otimes{\mathbb{C}}[Hyp] $ be the vertex algebra associated to the lattice $Hyp$,
where ${\mathbb{C}}[Hyp] $ is the group algebra of $Hyp$.

In this paper we shall consider the lattice $L= {\Z} c + {\Z} d$ such that $c= x$ and $d = 2 y$. Then $V_L$ is a  vertex subalgebra of $V_{Hyp}$.
 
Define the Heisenberg and the Virasoro vector:
\bea
I &= & - c_{L,I} c (-1), \label{Heis}  \\
\omega & = &       \frac{1}{2} c(-1) d (-1) + \frac{ c_L- 2 }{24 } c(-2) -   \frac{1}{2}  d (-2). \label{Vir}
\eea

Then the components of the fields
$$
I(z)=Y(I,z)=\sum_{n\in{\mathbb{Z}}}I(n)z^{-n-1},\quad L(z)=Y(\omega
,z)=\sum_{n\in{\mathbb{Z}}}L(n)z^{-n-2}%
$$
satisfy the commutation relations for the twisted Heisenberg--Virasoro Lie
algebra $\HVir$  and  $I$ and $\omega$ generate the simple Heisenberg--Virasoro vertex
algebra $L^{\HVir} (c_{L},0,c_{L,I},0,0)$ which we shall denote by $L^{\mathcal{H}
}(c_{L},c_{L,I})$.

We have:
\bea
\left[L(n), c(m) \right] &=& -m c(n+m) + (m ^2 -m)  \delta_{m+n,0} \\ 
\left[L(n),  d  (m) \right]   & = &   - m  d  (n+m) - \frac{c_L - 2} {12} (m ^2 -m)  \delta_{m+n,0}
\eea

%Let $V_{L} = M(1) \otimes{\mathbb{C}}[L] $ be the vertex algebra associated to the lattice $L$,
%where ${\mathbb{C}}[L] $ is the group algebra of $L$. 
Let $u = e^{c}$. Then  by \cite[Lemma 2.4]{AR1}
$$
Q= \mbox{Res}_{z} Y(u,z) = u_{0}.
$$
is a screening operator. This means that
$$[Q, I(n)] = [Q, L(n)] = 0 \quad \forall n \in {\Z},$$
i.e., $Q$ commutes with the action   of the Heisenberg-Virasoro algebra.

Recall that screening operators provide an important tool for construction of singular vectors in free-field realizations (cf.\ \cite{TK}).

One can show that
$$
\Pi(0)=M(1)\otimes{\mathbb{C}}[\Z c ]
$$
is a simple vertex algebra (cf.\ \cite{BDT}). Let $Y(\cdot, z) $ be the associated vertex operator. The vertex
operator $Y(u,z)$ and the screening operator  $Q$ are well-defined on every
$\Pi(0)$--module.

%{\color{red}  sve dijelimo s 2. Pogledaj kako izgledaju sve formule dolje}

Let $$ d^1 =  d + \frac{c_L-26} {12} c, \quad d^2 =  d -\frac{c_L-26} {12 } c. $$
Consider the following  irreducible $\Pi(0)$--modules  $$\Pi(p,r):= \Pi(0)\cdot e ^{ \frac{p-1}{2}  d^2 +   \frac{ 1-r}{2}   c}  \quad \mbox{where} \quad p  \in  \Z, r \in {\C} . $$
(The irreducibility of $\Pi(p,r)$ was also proved in \cite{BDT}).

Let
\bea && h_{p,r} =  (1-p^2) \frac{c_L-26}{24} + 1-p + (1-r) \frac{p}{2}. \nonumber \\
 && W_{p,r } = U(\HVir  )\cdot e^{\frac{p-1}{2} d^2 + \frac{1-r}{2} c }  \subset  \Pi(p,r) . \label{mod-simple} \eea
Set $v_{p,r} = e^{\frac{p-1}{2} d^2 + \frac{1-r}{2} c } $. Note that $h_{p,r+2} = h_{p,r}-p$.
 
The following result is proved in Propositions 2.5 and 2.7 of  \cite{AR1}.
\begin{proposition} \label{kratka-verzija}
We have:
\item[(1)] $W_{p,r } \cong V^{\HVir} (h_{p,r}, (1-p) c_{L, I} )$ if and only if $ p\in {\C} \setminus {\Z}_{\ge 1}$.
\item[(2)] $W_{p,r } \cong L^{\HVir} (h_{p,r}, (1-p) c_{L, I} )$ if  $ p\in  {\Z}_{\ge 1}$.
\item[(3)] For every  $r \in {\C}$,  $W_{0,r}\cong L^{\HVir} (\frac{c_L -2}{24},  c_{L, I} )$.
\end{proposition}
 
The following modules were not constructed in \cite{AR1}:
\begin{itemize}
\item Reducible Verma modules $V^{\HVir} (h_{p,r}, (1-p) c_{L, I} )$ in the case  $ p\in  {\Z}_{\ge 1}$.
\item Self-dual modules  $L^{\HVir} (h, c_{L, I} )$ for $h \ne \frac{c_L-2}{24}$.
\end{itemize}
We shall present the construction of these modules in Sections \ref{verma} and \ref{self-dual}.

We shall also need the following result.
\begin{lemma} \label{sub-sing}
Assume that $p\in\N$. As an ${\HVir}$--module, $\Pi(p,r)$ is generated by a family of ${\HVir}$--singular vectors $\{ v_{p,r-2\ell}  \ \vert \ \ell \in {\Z} \} $ and a family of $\HVir$--cosingular vectors
$\{ v_{p,r-2\ell } ^{(m)}  \ \vert \ \ell, m  \in {\Z}, \ m \ge 1 \}, $ where
 $$ v_{p,r-2\ell} ^{(m)} = \frac{(-1)^m}{ 2^m m!}  d(-p) ^m v_{p,r-2\ell} $$% e^{\frac{p-1}{2} d^2 + \frac{1-r}{2} c + \ell c }. $$ 
\end{lemma}
\begin{proof}
By \cite[Proposition 2.7]{AR1} we have that $\Pi(p,r)$ is generated  as ${\HVir}$--module by a family of singular vectors $\{ v_{p,r-2\ell} \ \vert \ell \in {\Z} \}$ and by a family of cosingular vectors $\{w^ {(m)} _{p, r-2\ell }  \ \vert \ \ell, m \in {\Z}, \ m \ge 1 \}$ satisfying
\bea
\label{Qm}
Q^ m w^ {(m)} _{p, r-2\ell }   = v_{p, r-2(\ell+m)}.
\eea
Let us prove that for cosingular vectors $w^ {(m)} _{p, r-2\ell }$ we may choose $v_{p,r-2\ell } ^{(m)}$. 

Recall the commutator relations in $\Pi (0)$:
\bea
[d(-k), e^c _j ]=2e^c_{j-k}, \qquad [e^c_j,c(-k)]=0, \qquad [e^c_0,e^c_j]=0. \nonumber
\eea
By induction on $m$ we see that $e^c_{-kp}v_{p,r-2\ell}^{(m)}$ is spanned by $S_{(k+i-1)p}(c)v_{p,r-2(\ell+1)}^{(m-i)}$, $ i=k-1,\dots m$ for each $k\in\Zp$. %We get that  $Q v_{p,r-2\ell}^{(m)}$ is a linear combination of $S_{(i-1)p}(c) v_{p,r-2(\ell +1)}^{(m-i)}$, $i=1,\dots m$. 
%{\color{red}Malo sam izmijenio ovu recenicu da bude detaljnije. Samo tehnicka stvar.} 
Therefore we have
\bea
\label{Q-sub-sing} Q^{j} v_{p,r-2\ell}^{(m)} = 
\begin{cases}
v_{p,r-2(\ell +j)}^{(m-j)} \bmod  \Ker_{\Pi(p,r)}  Q^m, \quad &j<m,\\
v_{p,r-2(\ell +m)}, \quad &j=m.
\end{cases}
\eea
The proof follows.
\end{proof}

See Figure \ref{diag} in the Appendix \ref{app} for structure of $\Pi(0)$--module $\Pi (p,r)$.

\begin{remark}
Note that $\Pi (p,r) = \Pi (p,s)$  if and only if $r-s \in 2 \Z $.
Moreover, $\Pi (1,1) = \Pi(0)$ and $U(\HVir) v_{1,1} \cong L^{\HVir} (c_L,c_{L,I})$.
\end{remark}
 
\section{An extension of the vertex algebra $L^{\HVir} (c_L, c_{L,I})$}{An extension of the Heisenberg--Virasoro vertex algebra} \label{extension}
Now we study modules $\Pi(p,r)$ in more details. There is a vertex subalgebra of $\Pi(0)$ which can be treated as an extension of the vertex algebra $L^{\HVir} (c_L, c_{L,I})$:
$$\overline{\Pi(0)} =\Ker_{\Pi(0)}  Q. $$
In this section we obtain filtrations of $\Pi(p,r)$ by $\overline{\Pi(0)}$--modules such that the subquotients are irreducible over $\overline{\Pi(0)}$.

%As we shall see, modules $\Pi(p,r)$ carry a natural $\overline{\Pi(0)}$--structure. We obtain a filtration by $\overline{\Pi(0)}$--modules such that the subquotients are irreducible over $\overline{\Pi(0)}$.

\begin{proposition} \label{alg-1}
Let $  \Pi(0)  ^{(n)}  =\Ker_{\Pi(0)}  Q ^{n+1}$. Then we have
\begin{itemize}
\item[(1)]  $\overline{\Pi(0)}$ is a vertex subalgebra of $\Pi (0)$ which is as  $L^{\HVir} (c_L, c_{L,I})$--module isomorphic to
$$ \overline { \Pi (0) } = \bigoplus_ {n \in {\Z} }  W_{1,1-2n} \cong \bigoplus_ {n \in {\Z} }  L^{\HVir}  ( n, 0). $$
 
\item[(2)] For every $n \in {\Z}_{\ge 0}$,  $\Pi(0) ^{(n)}$ and $\Pi(0) ^{(n+1)} / \Pi(0) ^{(n)}$ are  $\overline{\Pi(0)}$--modules. Moreover we have 
 $$ \Pi(0) = \bigcup_{n \ge 0}  \Pi(0)  ^{(n)}, \quad  \Pi(0)  ^{(n)} \cdot \Pi(0)  ^{(m )} \subset \Pi(0)  ^{(n+m)}. $$

\end{itemize}
\end{proposition}
\begin{proof}
Since $Q$ is a screening operator, $\overline{\Pi(0)} =\Ker_{\Pi(0)}  Q$ is a vertex subalgebra of $\Pi(0)$. By using \cite[Proposition 2.7]{AR1} we get
\bea \overline{\Pi(0) }& = &  \bigoplus_ {n \in {\Z} }  L^{\HVir} (c_L, c_{L,I}) . e^{n c} \nonumber \\
& = &   \bigoplus_ {n \in {\Z} } W_{1,1-2n} . \nonumber \eea
The proof of assertion (2) is clear.
\end{proof}

Condition (2) from Proposition \ref{alg-1} shows that $\HVir$--modules (and $\overline{\Pi(0)}$--modules) $ \Pi(0)  ^{(n)}  $ give a $\Z_{\ge 0}$ filtration on the vertex algebra $\Pi(0)$. 
In the same way we can construct a filtration on certain $\Pi(0)$--modules.

\begin{theorem} \label{filtracija-1}
Assume that $p\in\N$. Let $\Pi(p,r) ^{(m )}   =\Ker_{\Pi(p,r) }  Q ^{m+1} $. Then we have 
\begin{itemize}
\item[(1)] $\Pi(p,r) \cong \bigcup_{m	 \ge 0}  \Pi(p,r)  ^{(m)} , \quad \Pi(0)  ^{(n)} \cdot \Pi(p,r)  ^{(m )} \subset \Pi(p,r) ^{(n+m)}$.
\item[(2)]  For every $m \in {\Z}_{\ge 0} $  $\Pi(p,r) ^{(m )} $   is $\overline{\Pi(0)}$--module and $ \Pi(p,r)  ^{(m)} \subset  \Pi(p,r)  ^{(m+1)}$.
\item[(3)] $ \Pi(p,r)  ^{(m+1)} /  \Pi(p,r)  ^{(m)}$ is an  irreducible $\overline{\Pi(0)}$--module which is as  $\HVir$--module isomorphic to 
 \bea
 \bigoplus_{n \in {\Z} } W_{p,r-2n}  \cong  \bigoplus_{n \in {\Z} }  L^{\HVir}(h_{p,r}+np,(1-p)c_{L,I}). \label{dec-3} \eea
\end{itemize}
\end{theorem}
\begin{proof}
The proof of assertions (1) and (2) is clear. The decomposition  (\ref{dec-3})  essentially  follows from \cite[Proposition 2.7]{AR1}. Let us prove the irreducibility result  in (3). It  suffices to prove that
$$ W_{1,1-2n} \cdot W_{p,r-2 \ell  } = W_{p, 1-2 (n+\ell )}. $$
Recall that
$$  W_{1,1-2n} = U({\HVir} )\cdot e^{nc}, \qquad  W_{p,r-2\ell }  = U ({\HVir})\cdot u_{p,r-2\ell }  \bmod \Pi(p,r)  ^{(m)} ,$$
where $u_{p,r-2\ell }$ is an $\HVir$--cosingular vector $v_{p,r-2(\ell-m-1)}^{(m+1)}$ such that $Q^{m+1} u_{p,r-2\ell } = v_{p,r-2\ell } $.

Since $e^{nc}_{k_0}  v_{p, r-2\ell} = v_{p, r-2(\ell +n)}$ for $k_0 = - n (p-1) -1$ we get that
$$ Q^{m+1} e^{nc}_{k_0}  u_{p, r-2\ell}  =v_{p, r-2(\ell +n)} \ne 0. $$
So $ e^{nc}_{k_0}  u_{p, r-2\ell}  + \Pi(p,r)  ^{(m)} $ generates  $W_{p, 1-2 (n+\ell )} $ in $\Pi(p,r)  ^{(m+1)} /  \Pi(p,r)  ^{(m)}$. The proof follows.
\end{proof}

Figure \ref{diag} in Appendix \ref{app} represents a portion of module $\Pi (p,r)$ with action of $Q$ and $e^c_{-p}$ on (sub)singular generators obtained in Lemma \ref{sub-sing}. Quotient module $\Pi (p,r)^{(m+1)} / \Pi (p,r)^{(m)}$ is a direct sum of "slices", each generated by $v_{p,r-2\ell}^{(m+1)} + \Pi (p,r)^{(m)}$ and isomorphic to $W_{p,r-2(\ell+m+1)}$.

\bigskip
Let us now consider modules $\Pi(-p,r)$. Recall that $h_{-p,r+2}=h_{-p,r}+p$.
\begin{theorem} \label{nedef}
\item[(1)] As an $\HVir$--module $\Pi(0,r)$ is isomorphic to $$\bigoplus_{n\in\Z}W_{0,r}\cong\bigoplus_{n\in\Z}L^{\HVir}\left(\frac{c_L-2}{24},c_{L,I}\right).$$

\item[(2)] Let $p \in \N$. Consider $Q$ as $\overline{\Pi(0)}$--endomorphism of $\Pi(-p,r)$, and let $\Pi(-p,r)^{(m)}=\Img Q^m$. Then for $m\in\N$ we have
\begin{itemize}
	\item[(a)] $\Pi(-p,r) \cong \Pi(-p,r)^{(m)}$ as $\overline{\Pi(0)}$--modules,
	\item[(b)] $\Pi(-p,r)^{(m)}/\Pi(-p,r)^{(m+1)}$ is an irreducible $\overline{\Pi(0)}$--module which is as an $\HVir$--module isomorphic to $$\bigoplus_{\ell\in\Z}L^{\HVir}(h_{-p,r}+\ell p,(1+p)c_{L,I}) $$
\end{itemize}
\end{theorem}

\begin{proof}
It was shown in \cite{AR1} that $\Pi(p,r) \cong \bigoplus_{\ell\in\Z} W_{p,r+2\ell}$ as $\HVir$--module when $p \notin \N $. Decomposition in (1) then follows from Proposition \ref{kratka-verzija} (3). Since
\bea\label{Q-nedef}
Qv_{-p,r}=S_{p}(c)v_{-p,r-2}
\eea
we see that $Q(W_{-p,r+2\ell})\subset W_{-p,r+2(\ell-1)}$ and since $Q$ commutes with the action of $\HVir$ we have $\Ker Q=0$. Therefore $\Pi(-p,r) \cong \Img Q = \Pi(-p,r)^{(1)}$, so claim (a) follows by iteration.
 
Let us prove assertion (b). It suffices to prove the claim for $m=0$. General statement then follows from (a). Recall that $W_{-p,r+2\ell}\cong V^{\HVir}(h_{-p,r+2\ell},(1+p)c_{L,I})$ and notice that $Qv_{-p,r+2(\ell+1)} = S_{p}(c)v_{-p,r+2\ell}$ is an $\HVir$--singular vector in $W_{-p,r+2\ell}$ which generates the maximal submodule. Now $U(\HVir)\cdot(v_{-p,r+2\ell} + \Img Q) \cong L^{\HVir}(h_{-p,r+2\ell},(1+p)c_{L,I}) $. This proves the decomposition in (b). Proof of irreducibility in (b) is similar to the proof of Theorem \ref{filtracija-1}.
\end{proof}
See Figure \ref{diag2} in Appendix \ref{app} for reference.

\section{Relations in \texorpdfstring{$\Pi(0)$}{PI(0)}--modules}\label{singv}

In this section we shall apply a relation in the vertex algebra $\Pi(0)$ on its modules and recover an explicit formula for a singular vector in the Verma module $V^{\HVir} (h, (1-p)c_{L,I})$, for $p \ge 1$.
We shall use this formula in Section \ref{verma} when we construct a free field realization of these Verma modules.

Recall first that the Schur polynomials $S_r(x_1, x_2, \dots)$   in variables $x_1, x_2, \dots$ are defined by 
\bea \exp\left( \sum_{n=1}^{\infty} \frac{x_n}{n}  y^n\right) = \sum_{r=1}^{\infty} S_r (x_1, x_2, \dots) y^r \label{schur-def} \eea
Then for any $\gamma \in L$ we set
$ S_r (\gamma):= S_r( \gamma(-1), \gamma(-2), \dots)$.

Let $$e^{-c}(z) = Y(e^{-c}, z) = \sum_{i \in \Z} e^{-c}_i z^{-i-1}. $$
By direct calculation we get
 
\bea  L(-2) e^{-c} = \frac{c_L- 26}{24} c(-2) e^{-c} - \frac{1}{2} L(-1) (d(-1) e^{-c}) . \label{relacija} \eea
 
 Let $$s =(  L(-2) -  \frac{c_L- 26}{24} c(-2)  ) e^{-c}. $$
 Then we get
 \begin{lemma}
 On every $\Pi(0)$--module we have
\bea \mathcal Q  &=& s_0=  \mbox{Res}_z Y(s,z)  \nonumber \\
   & = & \sum_{i = 0}^{\infty}\left(  L(-2-i ) e^{-c} _i + e^{-c}_{-i-1} L(-1 + i) \right) \nonumber \\
    & & -  \frac{c_L- 26}{24}  \sum_{i \in {\Z} } (i+1) c(-i-2 )  e^{-c} _i\nonumber \\
  & = & 0. \nonumber \eea
 \end{lemma}
 \begin{proof}
 The assertion follows from  (\ref{relacija}) and the fact that 
 $$ (L(-1) u )_0 = 0$$
 in every vertex operator algebra.
 \end{proof}
 
Now we shall see some consequences of the relation $\mathcal Q=0$ for irreducible $\HVir$--modules $L^ {\HVir} (h,h_I)$ such that $h_I = (1-p) c_{L,I}$, $p \in \N$, which are realized as $ W_{p,r}$, for $r \in \mathbb{C} $.

We have
\bea
0 & = &   \mathcal Q e^{\frac{p-1}{2} d^2 + \frac{1-r}{2} c }  \nonumber \\ 
 & = &   \left( \sum_{i = 0}^{\infty} (  L(-2-i ) e^{-c} _i   )   + e^{-c} _{-1}  L(-1)  + h _{p,r} e^{-c} _{-2} \right)  e^{\frac{p-1}{2} d^2 + \frac{1-r}{2} c }  \nonumber \\
 &  &  -  \frac{c_L- 26}{24} \left(  \sum_{i \in {\Z} } (i+1) c(-i-2 )  e^{-c} _i \right) e^{\frac{p-1}{2} d^2 + \frac{1-r}{2} c }  \nonumber \\
& = &   \left( \sum_{i = 0}^{p-2} (  L(-2-i ) e^{-c} _i   )   + e^{-c} _{-1}  L(-1)  + h _{p,r} e^{-c} _{-2} \right)  e^{\frac{p-1}{2} d^2 + \frac{1-r}{2} c }  \nonumber \\
 &  &  -  \frac{c_L- 26}{24} \left(  \sum_{i = -2 }  ^{p-2} (i+1) c(-i-2 )  e^{-c} _i \right) e^{\frac{p-1}{2} d^2 + \frac{1-r}{2} c }  \nonumber \\
 & = &  \left( \sum_{i = 1}^{p}    L(-i) e^{-c} _{i-2}   +   \left( h _{p,r} -1  +   \frac{c_L- 26}{24} (p-1) \right) e^{-c} _{-2}  \right. \nonumber  \\
  &  &   \left. -  \frac{c_L- 26}{24}  \sum_{i = -1 }  ^{p-2} (i+1) c(-i-2 )  e^{-c} _i  \right) e^{\frac{p-1}{2} d^2 + \frac{1-r}{2} c }  \nonumber  \\
  & = & \left(  \sum_{i = 1}^{p}    L(-i) S_{p-i} (-c)   +  \left( h _{p,r} -1  +   \frac{c_L- 26}{24} (p-1) \right)S_p (-c)   \right.\\
 &  &  \left. -  \frac{c_L- 26}{24} \sum_{i = 1}  ^{p} (i-1) c(-i )  S_{p-i} \right) e^{\frac{p-1}{2} d^2 + \frac{1- (r+2)}{2} c }  \nonumber  
\eea
 
In this way we have proved:
\begin{proposition} \label{singular-formula}
Let $p\in\N$. In $W_{p,r+2} $ we have:
\bea 0&=&  \left(  \left (\sum_{i = 1}^{p}    L(-i) S_{p-i} (-c)    \right )  +    \left( h _{p,r} -1  +   \frac{c_L- 26}{24} (p-1)  \right) S_p (-c) +\phantom{x} \right.  \nonumber  \\
 &  &  \left. -  \frac{c_L- 26}{24} \left(  \sum_{i = 1}  ^{p} (i-1) c(-i )  S_{p-i} (-c)  \right)\right)  v_{p,r+2}  \nonumber \eea
 In particular, let $h_I = (1-p) c_{L,I} $, $h = h_{p, r+2}$. Then the singular vector of level $p$ in $V^{\HVir} (h, h_I)$ is  $\Phi_{p}  (L,c)\cdot v_{h,h_I}$ where 
\bea 
&&    \Phi_{p} (L, c)  :=  \sum_{i = 1}^{p}   \left(   L(-i) S_{p-i} (- c) \right)   +   S_p (-c)   \left( L(0)  +   \frac{c_L- 2}{24} (p-1)  \right )   \nonumber \\   &&  -  \frac{c_L- 26}{24} \left(  \sum_{i = 1}  ^{p} (i-1) c(-i )  S_{p-i} (-c)  ) \right)    \nonumber 
\eea
 \end{proposition}

\section{Deformed action of $\HVir$ on weight \texorpdfstring{$\Pi(0)$}{PI(0)}--modules and realization of Verma modules}
\label{verma}

As we noticed in Section \ref{review-ar}, the free field realization from \cite{AR1} does not provide realization of Verma modules $V^{\HVir} (h, (1-p)c_{L.I})$  and their singular vectors in the case $p \ge 1$. In order to understand these Verma modules, we shall use certain deformation of free field realization from \cite{AR1}.
We shall use the construction from \cite{AM-Selecta} to deform the action of the twisted Heisenberg--Virasoro algebra on $\Pi(0)$--modules (see also \cite{Bakalov}, \cite{Huang}). Let
$$\Delta (u, z) = z^{u_0} \exp  \left( \sum_{n=1} ^{\infty}  \frac{ u_n} {-n}  (-z) ^{-n} \right) . $$

First we recall a definition of  logarithmic modules. More information about structure theory of logarithmic modules can be found in literature on logarithmic vertex operator algebras  (see  \cite{AdM-2013}, \cite{CG},  \cite{HLL},  \cite{Mil}, \cite{Miy} and reference therein).
\begin{definition} 
\begin{itemize}
\item[(1)]   A module  $(M, Y_M)$ for the conformal vertex algebra with conformal vector $\omega$ is a logarithmic module of rank $m \in {\Z_{\ge 1} }$ if $$ (L(0)-L_{ss} (0) ) ^{m}  = 0, \quad (L(0)-L_{ss} (0) ) ^{m-1}   \ne 0,$$
where $L_{ss}(0)$ is the semisimple part of $L(0)$. 

\item[(2)] If  for every  $m \in {\Z_{\ge 1} }$  $(L(0)-L_{ss} (0) ) ^{m}  \ne 0$ on $M$, we say that $(M,Y_M)$ is a logarithmic module of infinite rank.
\end{itemize}
\end{definition}

\begin{theorem} \label{deformation}
For every $\Pi(0)$--module $(M,Y_M(\cdot,z))$,  $$(\widetilde{M}, \widetilde{Y}_{\widetilde{M} } (\cdot, z) ) := (M, Y_M (\Delta(u, z) \cdot, z))$$
is a $\overline{\Pi(0)} := \Ker_{\Pi(0)}  Q$--module. The action of Heisenberg--Virasoro algebra is
\bea
\widetilde{I(z)} & = &\sum_{n \in {\Z} } \widetilde{I(n)} z ^{-n-1}   =    \widetilde{Y}_{\widetilde{M} }  ( I,  z)  = I(z) \nonumber \\
\widetilde{L(z)} & = &  \sum_{n \in {\Z} } \widetilde{L(n)} z ^{-n-2}  =  \widetilde{Y}_{\widetilde{M} }  ( \omega,  z)  =  L(z) +  z ^{-1} Y_M (u ,z) .\nonumber  \eea
In particular,
\bea \label{deformed-action} \widetilde{I(n)} = I(n), \quad  \widetilde{L(n)} = L(n) + u_n. \eea
and $\widetilde{L(0)}  - \widetilde{L_{ss}(0)} = u_0 = Q$.
\end{theorem}

Recall the definition of $\Pi(0)$--modules $$\Pi(p,r):= \Pi(0)\cdot e ^{\frac{p-1}{2}   d^2 + \frac{1-r}{2} c} \quad \mbox{where} \quad p  \in  \Z, r \in {\C} . $$
Then  $\widetilde{\Pi(p,r)}$ are logarithmic ${\HVir}$--modules which are uniquely determined by the action (\ref{deformed-action}). 

We shall also consider the cyclic submodules: 
$$ \widetilde{ W_{p,r}} = U(\HVir )\cdot e^{\frac{p-1}{2} d^2 + \frac{1-r}{2} c }  \subset \widetilde{\Pi(p,r)}. $$

\subsection{Case $h_I=(1-p)c_{L,I}$}
We saw that for the undeformed action of ${\HVir}$ studied in \cite{AR1}, vector $v_{p,r}= e^{\frac{p-1}{2} d^2 + \frac{1-r}{2} c } $, for $ p \ge 1$, generates the irreducible highest weight module $W_{p,r}$. But we shall see below that $ \widetilde{ W_{p,r}}$ is isomorphic to a Verma module.  
 
 \begin{theorem} \label{verma-realization} Assume that $p \in \N$,  $ r \in {\C}$.  We have
 \item[(1)]  $\widetilde{W_{p,r}} \cong V^{\HVir} (h_{p,r}, (1-p)c_{L,I} )$.
  \item[(2)]  Singular vectors in $\widetilde{ W_{p,r}} \cong V^{\HVir} (h_{p,r}, (1-p)c_{L,I} )$ are
  $$ Sing= \{ v_{p,r-2n}  \ \vert \ n \ge 0\}.$$
 \end{theorem}
 \begin{proof}
It is clear that $ v_{p,r-2n}=  e^{\frac{p-1}{2} d^2 + \frac{1-r}{2} c + nc} $ is a singular vector for any $n \ge 0$. We only need to prove that 
\bea  v_{p,r-2n} \in \widetilde{ W_{p,r}}. \label{claim-chain}\eea
 Assume that  (\ref{claim-chain}) holds for $n \in {\Z}_{\ge 0}$. Since $e^c_ k v_{p,r-2n}  = 0$  for every $k \ge 1-p$ we have 
$$\widetilde{L(k)} = L(k) \quad \mbox{on} \ \  {\C}[c(-1), c(-2), \dots ] . v_{p,r-2n} .  $$
But since  $e^c_ {-p}  v_{p,r-2n}  =  v_{p,r-2(n+1) } $  we have
$$ \widetilde{L(-p)}  v_{p,r-2n}  = L(-p)   v_{p,r-2n}   +  v_{p,r-2(n+1) } . $$  
By using the expression for singular vector in  $ V^{\HVir} (h_{p,r}, (1-p)c_{L,I} )$ from Proposition \ref{singular-formula} we get for $h=h_{p,r-2n }$ 
\bea
 && \ \Phi_{p} (\widetilde L, c) \cdot  v_{p,r-2n}   \nonumber \\
=&&    \sum_{i = 1}^{p}   \left(  \widetilde  L(-i) S_{p-i} (- c) \right)   +   S_p (-c)   \left( h  +   \frac{c_L- 2}{24} (p-1)  \right )    \nonumber \\   &&  -  \frac{c_L- 26}{24} \left(  \sum_{i = 1}  ^{p} (i-1) c(-i )  S_{p-i} (-c)  ) \right)    v_{p,r-2n}  \nonumber \\
=&& \Phi_{p} (L, c) \cdot v_{p,r-2n}   +   v_{p,r-2(n+1)}  \nonumber \\
= &&   v_{p,r-2(n+1) } . \nonumber 
\eea
(Above we used  the fact that $  v_{p,r-2n}  $ generates the irreducible $\HVir$--module  $ L^{\HVir} (h_{p,r}, (1-p)c_{L,I} )$ for the undeformed action, so  $\Phi_{p} (L, c)\cdot v_{p,r-2n}  = 0$.)

Thus  we get that $v_{p,r-2(n+1)} = e^{\frac{p-1}{2} d^2 + \frac{1-r}{2} c +(n+1)c}$ belongs to $\widetilde{ W_{p,r}}$. The claim now follows by induction.
 \end{proof}
Finally, we obtain a deformed version of Theorem \ref{filtracija-1}.
\begin{theorem} \label{filtracija-2} 
Let $  Z  ^{(m )}   =\Ker_{ \widetilde {\Pi(p,r)} } Q ^{m+1} $. Then we have 
\begin{itemize}
\item[(1)] $\widetilde{ \Pi(p,r) }\cong \bigcup_{m \ge 0} Z  ^{(m)} ,  \quad \Pi(0)  ^{(n)} \cdot Z  ^{(m )}  \subset { Z  }  ^{(n+m)} $.
\item[(2)]  For every $m \in {\Z}_{\ge 0} $,  $ Z  ^{(m )} $ is a logarithmic  $\overline{\Pi(0)}$--module  of rank $m+1$ with respect to $\widetilde{L(0)}$.
\item[(3)] $ Z  ^{(m)} /  Z  ^{(m-1)}$ is a weight $\overline{\Pi(0)}$--module  which is as  $\HVir$--module isomorphic to 
 \bea   \bigcup_{n \in {\Z} } \widetilde {W_{p,r-2n}} . \label{dec-4} \eea
\end{itemize}
\end{theorem}
\begin{proof}

Assertion (1) is clear. Using relation (\ref{Q-sub-sing}) in the proof of Lemma \ref{sub-sing} we see that $v_{p,r-2\ell}^{(m)} \in Z^{(m)} \setminus Z^{(m-1)}$. 
Since
$\widetilde{L(0)}  - \widetilde{L_{ss}(0)} = Q$,  we have that   $  Z  ^{(m )} $ is a logarithmic module of $\widetilde{L(0)}$--nilpotent rank $m+1$ so (2) holds.
Assertion (3) results from following facts:

\begin{itemize}

\item [(a)] As an  $\HVir$--module $\widetilde{ \Pi(p,r) }$ is generated by set of vectors
 $$\{v_{p,r-2\ell  } \ \vert \ \ell  \in {\Z} \} \bigcup \{ v_{p,r-2\ell } ^{(m)}, \ \vert m, \ell  \in {\Z}, m \ge 1\}. $$
 
 \item [(b)]$ Z  ^{(m)} /  Z  ^{(m-1)}$ is a weight $\HVir$--module (i.e., non-logarithmic) generated by vectors $\{ v_{p,r-2 j }^{(m)}  + Z  ^{(m-1)} \vert \ \ell  \in {\Z} \}$.
 
 \item  [(c)] $v_{p,r-2\ell}^{(m)} + Z^{(m-1)}$ generates the Verma module  $\widetilde {W_{p,r-2\ell }}$.
 
\end{itemize}
Since $Q$ and $e^c_{-j}$ commute and by using (\ref{Qm}) we get
\bea
\nonumber Q^m e^c_{-j} v_{p,r-2\ell}^{(m)} = e^c_{-j} v_{p,r-2(\ell+m)} = S_{j-p}(c) v_{p,r-2(\ell+m)}
\eea
so $e^c_{-j} v_{p,r-2\ell}^{(m)} \in Z^{(m-1)}$ for $j<p$. Therefore 
\bea
\label{formula-3} &&
\widetilde{L(-j)}v_{p,r-2\ell}^{(m)} = 
\begin{cases}
L(-j)v_{p,r-2\ell}^{(m)} \bmod Z^{(m-1)}, &j<p,\\
\\
L(-p)v_{p,r-2\ell}^{(m)} + v_{p,r-2(\ell+1)}^{(m)}, &j=p.
\end{cases}
\eea
The proof of claims  (a) and (b)   easily follows from Theorem \ref{filtracija-1}  and (\ref{formula-3}). Let us prove claim (c).

We have  proved in  (\ref{formula-3})  that  $v_{p,r-2\ell}^{(m)} + Z^{(m-1)}$ is a highest weight vector with highest weight $(h_{p,r-2\ell}, (1-p) c_{L,I})$. Now, repeating the same arguments as in the proof of  Theorem \ref{verma-realization} we get
\bea
 && \Phi_{ p} (\widetilde L, c)\cdot  v_{p,r-2\ell}^{(m)}    \nonumber \\
=&&  \Phi_{ p} ( L, c) \cdot  v_{p,r-2\ell}^{(m)}   +  v_{p,r-2(\ell+1)}^{(m)} \bmod Z^{(m-1) } \nonumber \\
= &&   v_{p,r-2(\ell+1)}^{(m)} \bmod Z^{(m-1) }.   \nonumber 
\eea
This implies that $v_{p,r-2\ell}^{(m)} + Z^{(m-1)}$ generates the Verma module  $\widetilde    {W_{p,r-2\ell} }$ which contains all Verma modules  $\widetilde  {W_{p,r-2(\ell + j)} }$, $ j \in {\Z}_{\ge 1}. $  This completes the proof.
(See also Figure \ref{def-diagram} in Appendix \ref{app} where one can follow steps in the proof.)
\end{proof}

\subsection{Case $h_I=c_{L,I}$}
Note that $\widetilde{\Pi(0, r)}$ is an ${\HVir}$--module on which $I(0)$ acts as multiplication by $c_{L,I}$.

In particular,
$\widetilde {W_{0,r} } $ is a ${\Zp}$--graded logarithmic $\HVir$--module whose lowest component is $$\widetilde {W_{0,r} }(0) := \mbox{span}_{\C} \{ v_{0,r-2\ell} \ \vert \ \ell \in \Zp \} . $$
Moreover, since 
$$\widetilde{L(0)} -  \frac{c_L-2}{24}  = Q, \quad  Q ^n v_{0,r} = v_{0,r-2n} $$ we conclude that $\widetilde {W_{0,r} } $  is a $\Zp$--graded logarithmic module of infinite rank. See Figure \ref{diag-0} in Appendix \ref{app}.

\subsection{Case $h_I=(1+p)c_{L,I}$}
We saw that $\Pi(-p,r)$ contains a descending chain of submodules $\Pi(-p,r)^{(m)}$ isomorphic to $\Pi(-p,r)$ (Theorem \ref{nedef}). 

\begin{theorem} \label{p-neg} Let $p\in\N$. Then $\widetilde{\Pi(-p,r)}$ is a logarithmic  $\overline{\Pi(0)}$--module such that
$$ \dim {\C}[\widetilde{L(0)} ] v = \infty \quad \mbox{for every} \  v \in \widetilde{\Pi(-p,r)}. $$
Quotient $\widetilde{\Pi(-p,r)}/\widetilde{\Pi(-p,r)}^{(1)}$ is a weight module such that $$\widetilde{L(n)}v_{-p,r} = S_{p-n}(c)v_{-p,r-2} \ \    (1\leq n\leq p) \quad  \mbox{and} \quad \widetilde{L(n)}v_{-p,r}  = 0 \ \ (n > p). $$
\end{theorem}

\begin{proof}
Let  $S= \{ v_{-p,r-2\ell} \ \vert \ \ell \in \Z \}$. Let $\langle S \rangle $ be the $\HVir$--submodule generated by the set $S$. We shall  first prove that    $\widetilde{\Pi(-p,r)}  = \langle S \rangle $.  Since, as a vector space $\widetilde{\Pi(-p,r)} \cong \Pi(-p,r) \cong \bigoplus_{\ell\in\Z} W_{-p,r-2\ell}$,   it suffices to show that  $W_{-p,r-2\ell } \subset \langle S\rangle$ for each $\ell$.

 Take an arbitrary basis vector
$$ u = c(-p_1) \cdots c(-p_s) L(-q_1) \cdots L(-q_m)   v_{-p, r-2\ell} $$
of  $W_{-p,r-2\ell}$,  where  $\ell \in {\Z}  $, $p_1, \dots , p_s, q_1, \dots, q_m\ge 1$.
Then  by applying formula for  $\widetilde{L(n)}$  and  relation  $[e^c_{m},L(n)]=me^c_{m+n}$  we get
\bea &&     c(-p_1) \cdots c(-p_s)    \widetilde{L(-q_1)}  \cdots  \widetilde{L(-q_m )}  v_{-p, r-2\ell}   \nonumber \\
= &&  c(-p_1) \cdots c(-p_s) L(-q_1) \cdots L(-q_m)    v_{-p, r-2\ell}   + w \nonumber \eea
where  $w$ is a linear combination of vectors 
$$  c(-t_1) \cdots c(-t _{s'} ) L(-u_1) \cdots L (-u_{m'} )v_{-p, r-2\ell ' }  $$
such that $\ell '  \in {\Z} $, 
$ m' < m $ . The assertion now follows by induction on $m$.

Furthermore, we have
$$(\widetilde{L(0)}-L(0))^n v_{-p,r} = Q^n v_{-p,r} = (S_{p}(c))^n v_{-p,r-2n}.$$ Since $Q$ commutes with action of $\HVir$, we proved the first claim.

Taking a quotient by $\widetilde{\Pi(-p,r)}^{(1)} = \Img Q$ results in a weight module (i.e.,\ non logarithmic module) on which $\widetilde{L(0)} \equiv L(0)$.
\end{proof}
See  Figure \ref{diag-1} in Appendix \ref{app}.
%on pg.\ \pageref{diag-1}.

\begin{remark}
As far as we know, modules $\widetilde{\Pi(-p,r)}/\widetilde{\Pi(-p,r)}^{(1)}$, and their cyclic submodules generated by images of $v_{-p,r-2\ell}$ are weight $\HVir$--modules which haven't been analysed in the literature.
\end{remark}

\section{Realization of self-dual modules via Whittaker \texorpdfstring{$\Pi(0)$}{PI(0)}--modules}\label{whittaker}
In Section \ref{verma} we slightly refined the free field realization from \cite{AR1}, but these results still don't give a realization of all irreducible self-dual modules. In order to construct all self-dual modules we shall apply the deformation from Section \ref{verma} on the Whittaker $\Pi(0)$--module $\Pi_{\lambda}$ which was constructed in \cite[Section 11]{ALZ} and used for a realization of Whittaker $A_1 ^{(1)}$--modules at the critical level. 
As a by-product we shall see that self-dual modules for $\HVir$ have non-trivial self-extensions which are logarithmic modules.
\label{self-dual}

\subsection{Whittaker $\Pi(0)$--module $\Pi_{\lambda}$ }

We shall recall the construction of a Whittaker $\Pi(0)$--module $\Pi_{\lambda}$  from \cite[Section 11]{ALZ}.
Let $u = e^c$, $u ^{-1} = e ^{-c}$.
\begin{theorem} \label{ired-pi}  \cite{ALZ} Assume that $\lambda \ne 0$.
There is a $\Pi(0)$--module $\Pi_{\lambda}$ generated by the cyclic vector $w_{\lambda}$ such that $c(0) = - \mbox{Id}$ on $\Pi_{\lambda}$ and 
$$ u_0  w_{\lambda} = \lambda w_{\lambda},  \quad
u^{-1} _0  w_{\lambda} = \frac{1}{\lambda}  w_{\lambda}, \quad \ u_n w_{\lambda} =u^{-1} _n w_{\lambda} = 0  (n \ge 1). $$
As a vector space
$$\Pi_{\lambda} \cong {\C} [d(-n), c(-n-1) \  \vert n \ge 0 ] = {\C}[d(0)] \otimes M (1).$$

$\Pi_{\lambda}$ is ${\Zp}$--graded
 $$\Pi_{\lambda} = \bigoplus_{ n \in {\Z}_{\ge 0} } \Pi_{\lambda} (n) $$
 and lowest component is isomorphic to ${\C}[d(0)]$.
\end{theorem}

Recall also (cf.\ \cite{ALZ})  that the lowest component $\Pi_{\lambda} (0) $ is an irreducible Whittaker module for the associative algebra $\mathcal{A}$ defined by generators
$$ d(0) , e^{n c}  \quad (n \in {\Z})$$
and relations
$$ [ d(0) , e^{n c} ] = 2 n  e^{n c} , \ e^{n c} e^{m  c} = e^{( n + m) c}  \quad (n,m \in \Z). $$

\subsection{Realization of self-dual modules}

Now we can apply Theorem \ref{deformation} on the Whittaker $\Pi(0)$--module $\Pi_{\lambda}$. We get $\HVir$--module $\widetilde \Pi_{\lambda}$ , which is as a vector space isomorphic to $\Pi_{\lambda}$ and the (deformed) action of $\HVir$ is as follows:
\bea  \widetilde{I(0)}  \equiv - c_{L,I} c(0) \equiv  c_{L,I} \mbox{Id} \quad \mbox{on} \ \ \widetilde  \Pi_{\lambda},  \eea
and on the lowest component  $\widetilde \Pi_{\lambda} (0)$ we have
\bea \widetilde{L(0)}   &\equiv & \frac{1}{2} d(0) ( c(0) +1)   -  \frac{c_L -2}{24}  c(0) +  u_0 \quad  \mbox{on} \  \  \widetilde \Pi_{\lambda} (0) \nonumber \\
 & \equiv &   \frac{c_L -2}{24}  \mbox{Id}  +  u_0 \quad  \mbox{on} \  \  \Pi_{\lambda} (0). \label{for-L0} \eea
 This implies:
\bea
&& \label{for-hw} \widetilde{I(0)}  w_{\lambda}  = c_{L,I}  w_{\lambda} , \quad \widetilde{L(0)} w_{\lambda}  = \left(\frac{c_L-2}{24}  + \lambda \right) w_{\lambda}.\eea

Define also the following (logarithmic)  cyclic module:
 $$ \widetilde \Pi_{\lambda} ^{(n)} = U(\HVir)\cdot d(0) ^ n w_{\lambda}. $$

\begin{lemma} \label{cyclic-n}
 We have:
$$ \widetilde \Pi_{\lambda}^{ (n+1)} \supset  \widetilde \Pi_{\lambda} ^{(n)} , \quad \widetilde \Pi_{\lambda} = \bigcup_{n\in {\Zp} } \widetilde  \Pi_{\lambda} ^{ (n)}.  $$
\end{lemma}
\begin{proof}  
By using (\ref{for-L0}) one can easily see that for $0 \le m \le n$ there is a polynomial $P(x)$ such that
$$  P(  \widetilde{L(0)} ) d(0) ^n w_{\lambda}   = d(0) ^{m }  w_{\lambda}.  $$
This proves that $ \widetilde \Pi_{\lambda} ^{ (n+1)} \supset  \widetilde \Pi_{\lambda} ^{(n)} $ for $n \in {\Z}_{\ge 0}$.

Take an arbitrary basis vector
$$ u = c(-p_1) \cdots c(-p_s) d(-q_1) \cdots d(-q_r)   d(0) ^{\ell} w_{\lambda}$$
of $\widetilde \Pi_{\lambda} $, where  $\ell \in {\Z} _{\ge 0}$, $p_1, \dots , p_s, q_1, \dots, q_r \ge 1$.
Then
\bea &&     c(-p_1) \cdots c(-p_s)    \widetilde{L(-q_1)}  \cdots  \widetilde{L(-q_r )}  d(0) ^{\ell} w_{\lambda}  \nonumber \\
= && A c(-p_1) \cdots c(-p_s) d(-q_1) \cdots d(-q_r) d(0) ^{\ell } w_{\lambda} + w \nonumber \eea
where $A\ne 0$ and  $w$ is a linear combination of vectors 
$$  c(-t_1) \cdots c(-t _{s'} ) d(-u_1) \cdots d(-u_{r'} ) d(0) ^{\ell ' } w_{\lambda}$$
such that $\ell '  \in {\Z}_{\ge 0}$, 
$ r' < r $  or $r= r' $ and   $ u_1+ \cdots + u_{r'} < q_1 + \cdots +q_r$. The assertion now follows by induction.
\end{proof}

\begin{theorem} \label{self-dual-real} 
For every $\lambda \in {\C}$, $\lambda \ne 0$ we have:
\item[(1)] $\widetilde \Pi_{\lambda}$ is a logarithmic $\HVir$--module of infinite rank with respect to $\widetilde{L(0)}$.

\item[(2)] $\widetilde \Pi_{\lambda}^{(0)} $ is an irreducible self-dual $\HVir$--module with highest weight $$ (h,h_I) =(\frac{c_L-2}{24}  + \lambda, c_{L,I}  ). $$ 

\item[(3)]  $\HVir$--module $\widetilde \Pi_{\lambda}$ admits the $\Zp$--filtration:
$$\widetilde \Pi_{\lambda} = \cup_{n\in {\Zp} } \widetilde  \Pi_{\lambda} ^{ (n)} $$ such that
$$ \widetilde  \Pi_{\lambda} ^{ (0 )} =L^{\HVir} (h,h_I) , \quad  \widetilde \Pi_{\lambda} ^{ (n+1  )}  /  \widetilde  \Pi_{\lambda} ^{ (n)}  \cong L^{\HVir} (h,h_I) . $$
Every $\widetilde \Pi_{\lambda} ^{(n)}$ is a logarithmic $\HVir$--module of rank $n$  with respect to $\widetilde{L(0)}$.
\end{theorem}
 \begin{proof} 
 
 (1) follows from  the fact that on the top component $\widetilde \Pi_{\lambda} (0)$ we have $ Q = u_0 =  \widetilde{L(0)}  - \frac{c_L -2}{24}$.
 
 (2) Using  (\ref{for-hw}) and the fact that the Verma module with highest weight  $ (h, h_I) =  (\frac{c_L-2}{24}  + \lambda, c_{L,I} ) $ is irreducible we get $ L^{\HVir} (h,h_I) = U(\HVir)\cdot w_{\lambda} = \widetilde \Pi_{\lambda}^{(0)} $.  
 
 (3) First we notice that for $m \ge 0$ we have  $$ \widetilde{L(m)} d(0) ^{n+1} w_{\lambda} = h \delta_{m,0}   d(0) ^{n+1} w_{\lambda}  \quad \mbox{mod} \ \widetilde \Pi_{\lambda} ^{(n)}.$$
Therefore we have isomorphism $  L^{\HVir} (h,h_I) \rightarrow   \widetilde \Pi_{\lambda} ^{ (n+1  )}  /  \widetilde  \Pi_{\lambda} ^{ (n)} . $
The proof now follows from Lemma \ref{cyclic-n}.
\end{proof}
See Figure \ref{diag-lambda} in Appendix \ref{app} for reference.

 We list two interesting consequences of previous theorem.
 \begin{corollary} Logarithmic $\HVir$--module $\widetilde \Pi_{\lambda} ^{(1)} $ is a non-split self-extension of irreducible self-dual module $L^{\HVir} (h, h_I)$:
 $$ 0 \rightarrow  L^{\HVir} (h, h_I) \rightarrow  \widetilde \Pi_{\lambda} ^{(1)} \rightarrow  L^{\HVir} (h, h_I) \rightarrow 0. $$
 \end{corollary}
 
 Note that the vertex algebra $\overline{\Pi(0)}$ is not ${\Z}_{\ge 0}$--graded since  for every $n \in {\Z}$, $e^{nc}$ has weight $n$. Irreducible  $\overline{\Pi(0)}$--modules from Theorem  \ref{filtracija-2} (3)  are not $\Z_{\ge 0}$--graded. But, quite surprisingly, the vertex algebra $\overline{\Pi(0)}$ admits a large family of ${\Z}_{\ge 0}$--graded modules which are self-dual. We also construct a family of intertwining operators which haven't appeared in \cite{AR1}.
 \begin{corollary}  We have:
 
 \item[(1)]  $L^{\HVir} (h, c_{L,I}) $ is an irreducible $\Z_{\ge 0}$--graded $\overline{\Pi(0)}$--module.
 
 \item[(2)] For every $n \in {\Z}$ there is a non-trivial intertwining operator of type
 $$ {L^{\HVir} (h, c_{L,I}) \choose L^{\HVir} (n ,  0 )   \ \  L^{\HVir} (h, c_{L,I}) }. $$
\end{corollary}

 \section{Some applications to the $W(2,2)$--algebra}
 In \cite{AR1} we introduced a free field realization of the $W(2,2)$-algebra as a subalgebra of the Heisenberg Virasoro algebra.
 
Recall that $W(2,2)$ is realized as a subalgebra of $L^{\HVir} (c_L, c_{L,I})$ generated by $L(z)$ and
$$ W(z) = c_{L,I} ^2   \overline W(z) $$ where $$\overline W (z) =  \sum_{n \in {\Z}} \overline W(n) z^{-n-2} = Y( c(-1) ^2 - 2 c(-2), z) = c(z) ^2 - 2 \partial c(z). $$
   In the paper  \cite{AR2} we discussed a free field realization of highest weight $W(2,2)$--modules.  
We constructed  in  \cite[Section 4]{AR2} $W(2,2)$--homomorphism
$ S_1 :  L^{\HVir} (c_L, c_{L,I}) \rightarrow L^{\HVir} (1, 0) $
such that $\mbox{Ker} _{  L^{\HVir} (c_L, c_{L,I}) } S_1 $ is the simple vertex algebra $L^{W(2,2)} (c_L, -24 c_{L,I} ^2 ). $
In this paper  shall present  a bosonic, non-local expression for the screening operator   $S_1$.
 
The vertex algebra $W(2,2)$ has appeared in physics literature as  the Galilean Virasoro algebra (\cite{RR}, \cite{BG}, \cite{BGMM}) and as $\text{BMS}_3$ algebra (\cite{BJMN}). We noticed a free field realization of the $W(2,2)$-algebra in terms of the $\beta \gamma$ systems in \cite{BJMN}. We shall see how this realization relates to our approach.

\subsection{A bosonic formula for the second screening operator and $W(2,2)$--algebra}

Our approach is motived by the realization of screening operators in LCFT from \cite{AdM-2010} and \cite{AdM-2013}.
 Recall the definition of modules $W_{p,r}$ from (\ref{mod-simple}).  For $r \in {\Z}$ we define:
 \bea  S 
  & =&  - \mbox{Res}_{z} \mbox{Res}_{z_1} \left( \mbox{Log} (1 -\frac{z_1}{z} ) e^{c}(z) d^1 (z_1)  -  \mbox{Log} (1 -\frac{z}{z_1} ) d^1 (z_1) e^{c}(z) \right) \nonumber  \\
 &=&  \sum_{j=1} ^{\infty} \frac{1}{j} \left( d^1 (-j) e^{c}_j - e^{c}_{-j} d^1 (j) \right) : W_{p,r} \rightarrow W_{p,r-2}.  \label{formula-screening-2} \eea
 
 \begin{lemma}  \label{komutator-1}
 We have:
 \bea
 && [L(m), S] = d^1  (m ) e^c _0  -e^c _m d^1 (0)  +2\delta_{m,0}e^c_0  ,
 \nonumber \\  
 &&  [c(m), S ]  =  2  e^c_m - 2 \delta_{m,0} e^c _0 ,  [W(m), S] = 0.  \nonumber 
  \eea
 \end{lemma}
 \begin{proof} In the proof we use the following formulas
  \bea
\left[L(n), e^c_m \right] &=& -m e ^{c}_{n+m}, \label{com-ec}  \\ 
\left[L(n), d^1 (m) \right]   & = &   - m d^1  (n+m) - 2  (m ^2 -m)  \delta_{m+n,0}   \label{com-ld1}.
\eea
We have: 
 \bea
 [L(n), S ] & =&    \sum_{j=1} ^{\infty} \frac{ - j }{j} \left( d^1 (-j) e^{c}_{j+n  } +  e^{c}_{-j+n  } d^1 (j) \right)  \nonumber \\
  & & +  \sum_{j=1} ^{\infty} \frac{  j }{j} \left( d^1 (-j +n ) e^{c}_{j + n } +  e^{c}_{-j  } d^1 (j +n ) \right)  \nonumber \\
   & & -2   \sum_{j=1} ^{\infty} \frac{  j ^2 + j }{j} \delta_{- j+n,0}  \  e^{c}_{j + n }    
     + 2   \sum_{j=1} ^{\infty} \frac{  j ^2 - j }{j} \delta_{ j+n,0}  \  e^{c}_{- j   }    \nonumber \\
     &= & -   \sum_{j=1} ^{\infty}   \left( d^1 (-j) e^{c}_{j+n  } +  e^{c}_{-j+n  } d^1 (j) \right)  +  \sum_{j=1} ^{\infty}  \left( d^1 (-j +n ) e^{c}_{j   } +  e^{c}_{-j  } d^1 (j +n ) \right)  \nonumber \\
   & & -2   \sum_{j=1} ^{\infty}  (j+1) \delta_{- j+n,0}  \  e^{c}_{j   }    + 2   \sum_{j=1} ^{\infty}  (j-1) \delta_{ j+n,0}  \  e^{c}_{- j   }    \nonumber \\
     &=& - 2 (n+1) e^c_{n } +  2 \delta_{n,0} e^c _0 + d^1(0) e^{c}_{n } + \cdots + d^1(n-1) e^c_{ 1} \nonumber \\ 
     && - \left(   e^c_{n-1} d^1 (1) + \cdots + e^{c}_{1} d^1 (n-1)      \right) - e^c _0 d^1 (n) \nonumber \\
         &=&  -  d^1  (n) e^c _0   +  e^c _{ n} d^1 (0) + 2 \delta_{n,0} e^c _0.  \nonumber   
 \eea
 Relation $[c(m), S] = 2 ( 1- \delta_{m,0}) e^c _m$ follows directly from the definition of the operator $S$.
 Next we  have
 \bea
 [\overline W(n), S] & = &  \left( \sum_{k \in {\Z} } [c(k) c(n-k), S]\right)  + 2 n [c(n-1), S] \nonumber \\
&=&  - \sum_{k \in {\Z} } ( c(k) e^c_{n-k} + c(n-k) e^c _k )  -  2 n e^c _ {n-1} \nonumber \\
&=& -4 (D e^c ) _{n} - 4 n e^{c} _{n-1}  = 0 \nonumber 
 \eea
 The proof follows.
 \end{proof}

 Now, we will see that in the case $r=1$ our operator $S$ is a multiple of the screening operator $S_1$ from \cite{AR2}:
  \begin{corollary} Let $r=1$, and consider $S:  W_{p,1} \rightarrow W_{p,-1}. $
 Then  $S$ commutes with the action of  the $W(2,2)$-algebra:
$$  [S, W(n)] = [ S, L(n) ] = 0 \quad (n \in {\Z}  ). $$
Moreover, $S$ is a $W(2,2)$--homomorphism which is proportional to $S_1$.
 \end{corollary}
\begin{proof}
In the case $r=1$ we have that $d^1 (0) $ and $Q=e^c_0$ act trivially on $W_{p,1}$, and therefore Lemma \ref{komutator-1} implies that $$  [S, W(n)] = [ S, L(n) ] = 0 \quad (n \in {\Z}  ). $$
It is clear that the $W(2,2)$--homomorphism  $S_1 :  W_{p,1} \rightarrow W_{p,-1} $ from \cite[Section 4]{AR2} is uniquely determined by the properties
$$[S_1, L(n)] = 0,\qquad [S_1, I (n)] =  - e^c_n, $$
which gives $[S_1, c(n)] =\frac{1}{c_{L,I} } e^{c} _n$.
Now  Lemma \ref{komutator-1}  gives that $S= 2 c_{L,I}  S_1$.
\end{proof}

\subsection{On the Banerjee, Jatkar, Mukhi, Neogi's  free field realization of the $\text{BMS}_3$--algebra }
Recently, Banerjee, Jatkar, Mukhi and Neogi in \cite{BJMN} have discovered a new free field realization of the $W(2,2)$--algebra for central charge $c_L= 26$. The vertex algebra $L^{W(2,2)} (26, c_W)$ is realized inside of the $\beta \gamma$ system. Since the $\beta \gamma $--system can be embedded into the vertex algebra $\Pi(0)$, one may try to extend this realization in order to obtain an arbitrary central charge $c_L$. Quite surprisingly, even in the case of the larger vertex algebra $\Pi(0)$, one gets the $W(2,2)$--structure only for $c_L=26$.

Recall the definition of following Virasoro vector of central charge $c_L \in {\C}$.
\bea
\omega & = &       \frac{1}{2} c(-1) d(-1) + \frac{ c_L- 2 }{24 } c(-2) - \frac{1}{2}  d(-2). \label{Vir1}
\eea
 
 We shall now deform this vector in a different way:
\begin{lemma} \label{vir-1}
For every $\mu \in {\C}$
$$\widetilde \omega =  \omega + \mu e^{-c} _{-4} {\bf 1}  =  \omega +   \frac{\mu }{6} D^3 e^{-c}.$$
is a Virasoro vector of central charge $c_L$.
\end{lemma}
 
The proof of lemma follows from a more general statement (which is also noticed in \cite{BJMN}):

Claim: {\em Assume that $(V, Y, {\bf 1}, \omega)$ is a VOA of central charge $c$, and $v$ is a  primary, commutative  vector of conformal weight $-1$. Then $\widetilde \omega =  \omega + \frac{1 }{6} D^3 v$ 
is a Virasoro vector of central charge $c$.}
 
 \vskip 5mm
 
 The following result is obtained in \cite[Section 2]{BJMN}. We include a proof of this result from which one can see that such construction works only for $c_L=26$.
 \begin{proposition} \cite{BJMN}.
 The vertex algebra $L^{W(2,2)} (c_L, c_W)$ for $c_L =26$ is isomorphic to  a vertex subalgebra of $\Pi(0)$ generated by
 \bea
 \widetilde \omega &=&      \frac{1}{2} c(-1) d(-1) + \frac{ c_L- 2 }{24 } c(-2) - \frac{1}{2}  d(-2) +   \frac{\mu }{6} D^3 e^{-c} \nonumber \\
 w & = & (d(-1) + \frac{c_L-14}{12} c(-1) ) e^{c} \nonumber 
 \eea
 where
 $$\mu = -\frac{c_W}{4}. $$
 \end{proposition}
\begin{proof} 
By direct calculation we get
\bea
w_0 w &=& \frac{c_L- 26}{3} c(-1) e^{2c}  =  \frac{c_L- 26}{6}  D e^{2c}     \label{ww1} \label{for-w1}\\
w_1 w &=&  \frac{c_L- 26}{3}  e^{2c}   \label{ww2} \label{for-w2} \\
w_n  w &=&   0 \quad (n \ge 2).   \label{for-w3} 
\eea
By using formulas
\bea
\left[  L(n), c(m) \right] &=& -m c(n+m) + (m ^2 -m)  \delta_{m+n,0} \label{com-lc}  \\ 
\left[ L(n), d(m) \right]   & = &   - m d (n+m) - \frac{c_L - 2} {12} (m ^2 -m)  \delta_{m+n,0}   \label{com-ld}
\eea
we get that 
\bea
  L(1) w & = &\left(  [ L(1), d(-1)]  + \frac{c_L-14}{12} [  L(1),  c(-1)]  \right) e^c \nonumber \\
 &=  &\left( 2 -  \frac{c_L - 2} {6}  + \frac{c_L-14}{6} \right) e^c \nonumber \\
 & = & 0, \nonumber
\eea
which easily implies that 
\bea  && L(n) w = 2 \delta_{n,0} w \quad (n \ge 0). \eea
Since 
\bea  &&   \widetilde L(2) w = -  \mu  e^{-c} _0  w = -2 \mu =  \frac{c_W}{2}, \nonumber  \eea
we get
\bea  &&  \widetilde L(n) w = 2 \delta_{n,0}  w +  \frac{c_W}{2} \delta_{n,2}  \quad (n \ge 0). \label{primary-new}  \eea
Claim now follows from (\ref{for-w1})-(\ref{for-w3}), (\ref{primary-new}) and Lemma  \ref{vir-1}.
\end{proof}

\begin{remark}
The Weyl vertex algebra  (also called the $\beta \gamma$ system in the physical literature) can be realized as a subalgebra of the vertex algebra $\Pi(0)$ as follows:
\bea    \beta  = \left( c(-1) + d(-1) \right) e^{c},  \ \gamma = - \frac{1}{2} e^{-c}. \label{bosonization} \eea
Then  the components of the fields $Y(\beta, z) = \sum_{n \Z} \beta (n) z^{-n-1}$, $Y(\gamma, z) = \sum_{ n\in {\Z}} \gamma(n) z^{-n}$ satisfy the commutation relation for the Weyl algebra 
 $$[\beta(n), \beta(m)] = [\gamma(n), \gamma(m)] = 0, \quad [\beta(n), \gamma (m)] =\delta_{n+m,0}. $$
The vertex algebra $\Pi(0)$ can be treated as  a certain localization of the Weyl vertex algebra (for details see \cite{efren}, \cite{MSV}, \cite{ALZ}).

By using (\ref{bosonization}) we see that for $c_L=26$ $L^{W(2,2)} (c_L, c_W)$ is realized as a subalgebra of the Weyl vertex algebra, which  corresponds to the realization in   \cite{BJMN}.
\end{remark}
 
\begin{remark}
It would be interesting to investigate the structure of $W(2,2)$--modules $\Pi(p,r)$ with this new action. Since
$$ \widetilde{L(n)} = L(n) -\frac{1}{6} (n+1) n (n-1) \mu e^{-c}_{n-2},$$
we have that $\Pi(p,r)$ and $\Pi_{\lambda}$ are weight $L^{W(2,2)} (c_L, c_W)$--modules on which $\widetilde{L(0)} = L(0)$. In our forthcoming papers we plan to investigate the appearance of these modules in the fusion rules analysis at $c_L=26$.
\end{remark}

\appendix
\section{Figures}\label{app}
Here we present some visualizations of modules $\Pi(p,r)$ and $\widetilde{\Pi(p,r)}$.

\begin{figure}[H]
\begin{tikzpicture}[xscale=0.99,yscale=1.2]

\node at (0,0.75) {$\cdots$};
\node at (13.5,0.75) {$\cdots$};
\draw[thick] (0,0) --(3,3) --(6,0);
\draw[thick] (6.5,0) --(8.5,2) --(10.5,0);
\draw[thick] (11,0) --(12,1) --(13,0);
\draw[dashed, thin] (2,0) --(4,2);
\draw[dashed, thin] (4,0) --(5,1);
\draw[dashed, thin] (8.5,0) --(9.5,1);

\node[label=above:$v_{p,r+2}$,tokens=1] (a) at (3,3) {};
\node[label=above:$v_{p,r}$,tokens=1] (b) at (8.5,2) {};
\node[label=above:$\quad v_{p,r-2}$,tokens=1] (c) at (12,1) {};
\node[label=above:$\qquad v_{p,r+2}^{(1)}$,tokens=1] (a1) at (4,2) {};
\node[label=above:$\qquad v_{p,r+2}^{(2)}$,tokens=1] (a2) at (5,1) {};
\node[label=above:$\quad v_{p,r}^{(1)}$,tokens=1] (b1) at (9.5,1) {};

\path[->,dotted] (a) edge [bend left=10] node [above] {$e^c_{-p}$} (b)
			(b) edge [bend left=20] node [above] {$e^c_{-p}$} (c)
			(a1) edge [bend left=10] node {} (b1)
			(a1) edge 		      node [above left] {$Q$}    (b)
			(a2) edge 		      node [above left] {$Q$}    (b1)
			(b1) edge 		      node [above left] {$Q$}    (c);

\draw[opacity=0] (0,0) rectangle (4,2) node [midway,opacity=1,rotate=50] {$\Pi(p,r)^{(0)}=\Ker Q$};
\draw[opacity=0] (2,0) rectangle (5,1) node [midway,opacity=1,rotate=50] {$\frac{\Ker Q^2}{\Ker Q}$};
\draw[opacity=0] (4,0) rectangle (6,0) node [midway,opacity=1,rotate=50] {$\frac{\Ker Q^3}{\Ker Q^2}$};
\draw[opacity=0] (6.5,0) rectangle (9.5,1) node [midway,opacity=1,rotate=50] {$\Ker Q$};
\draw[opacity=0] (8.5,0) rectangle (10.5,0) node [midway,opacity=1,rotate=50] {$\frac{\Ker Q^2}{\Ker Q}$};
\draw[opacity=0] (11,0) rectangle (13,0) node [midway,opacity=1,rotate=50] {$\Ker Q$};
\end{tikzpicture}
\caption{$\overline{\Pi(0)}$--module $\Pi (p,r)$, $p\in \N$}
\label{diag}
\end{figure}

\begin{figure}[H]
\begin{tikzpicture}[xscale=0.8,yscale=0.98]

\draw[thick] (0,0) --(1,1) --(2,0);
\draw[thick] (2.5,0) --(5,2.5) --(7.5,0);
\draw[thick] (8,0) --(12,4) --(16,0);
\draw[dashed, thin] (4,0) --(5,1) --(6,0);
\draw[dashed, thin] (9.5,0) --(12,2.5) --(14.5,0);
\draw[dashed, thin] (11,0) --(12,1) --(13,0);

\node[label=above:$v_{-p,r+2}\quad$,tokens=1] (a) at (1,1) {};
\node[label=above:$v_{-p,r}$,tokens=1] (b) at (5,2.5) {};
\node[label=above:$\quad v_{-p,r-2}$,tokens=1] (c) at (12,4) {};
\node[tokens=1] (b1) at (5,1) {};
\node[label=below:$S_{p}(c)v_{-p,r}$,tokens=0] at (5.7,0.7) {};
\node[tokens=1] (c1) at (12,2.5) {};
\node[label=below:$S_{p}(c)v_{-p,r-2}$,tokens=0] at (12.8,2.1) {};
\node[tokens=1] (c2) at (12,1) {};
\node[label=below:$(S_{p}(c))^2v_{-p,r-2}$,tokens=0] at (13,0.7) {};
\node at (-0.25,0.75) {$\cdots$};
\node at (16.5,0.75) {$\cdots$};

\path[->,dotted] (a) edge [bend left=10] node [above] {$e^c_{p}$} (b)
			(b) edge [bend left=20] node [above] {$e^c_{p}$} (c)
			(a) edge 		      node [above left] {$Q$}    (b1)
			(b) edge 		      node [above left] {$Q$}    (c1)
			(b1) edge 		      node [above left] {$Q$}    (c2);

\end{tikzpicture}
\caption{$\overline{\Pi(0)}$--module $\Pi (-p,r)$, $p\in\N$ } 
\label{diag2}
\end{figure}

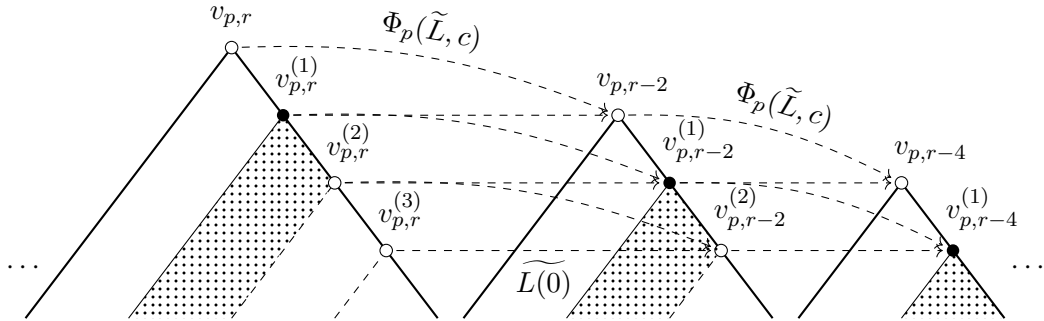
\begin{figure}[H]
\begin{tikzpicture}[xscale=0.685,yscale=0.9]

\node at (0,0.75) {$\cdots$};
\node at (19.5,0.75) {$\cdots$};
\draw[thick] (0,0) --(4,4) --(8,0);
\draw[thick] (8.5,0) --(11.5,3) --(14.5,0);
\draw[thick] (15,0) --(17,2) --(19,0);
\draw[thin] (2,0) --(5,3);
\draw[dashed, thin] (4,0) --(6,2);
\draw[dashed, thin] (6,0) --(7,1);
\draw[thin] (10.5,0) --(12.5,2);
\draw[dashed, thin] (12.5,0) --(13.5,1);
\draw[thin] (17,0) --(18,1);

\node[label=above:$\quad v_{p,r}^{(1)}$,tokens=1] (a1) at (5,3) {};
\node[label=above:$\qquad v_{p,r-2}^{(1)}$,tokens=1] (b1) at (12.5,2) {};
\node[label=above:$\qquad v_{p,r-4}^{(1)}$,tokens=1] (c1) at (18,1) {};
\tikzstyle{token}=                 [fill=white,draw,circle,
                                    inner sep=0.5pt,minimum size=1ex]
\node[label=above:$v_{p,r}$,tokens=1] (a) at (4,4) {};
\node[label=above:$\quad v_{p,r}^{(2)}$,tokens=1] (a2) at (6,2) {};
\node[label=above:$\quad v_{p,r}^{(3)}$,tokens=1] (a3) at (7,1) {};
\node[label=above:$\quad v_{p,r-2}$,tokens=1] (b) at (11.5,3) {};
\node[label=above:$\qquad v_{p,r-2}^{(2)}$,tokens=1] (b2) at (13.5,1) {};
\node[label=above:$\qquad v_{p,r-4}$,tokens=1] (c) at (17,2) {};

\path[->,dashed,thin] (a) edge [bend left=9] node [above,rotate=-10] {$\Phi_{ p} (\widetilde L, c)$} (b)
				(b) edge [bend left=11] node [above,rotate=-17] {\quad $\Phi_{ p} (\widetilde L, c)$} (c)
				(a2) edge [bend left=9] node {} (b2)
				(a1) edge node {} (b)
				(a2) edge node {} (b1)
				(a3) edge node [below] {$\widetilde {L(0)} \ \ $} (b2)
				(b1) edge node {} (c)
				(b2) edge node {} (c1)
				(a1) edge [bend left=9] node {} (b1)
				(b1) edge [bend left=11] node {} (c1);

\fill[pattern=dots] (2,0) to (5,3) to (6,2) to (4,0);
\fill[pattern=dots] (10.5,0) to (12.5,2) to (13.5,1) to (12.5,0);
\fill[pattern=dots] (17,0) to (18,1) to (19,0);
\end{tikzpicture}
\caption{Deformed action of $\HVir$ on $\widetilde{\Pi (p,r)}$, $p>0$. Dotted area represents a cyclic submodule of $\widetilde{\Pi (p,r)^{(1)}} / \widetilde{\Pi (p,r)^{(0)}}$ generated by $v_{p,r}^{(1)}$ which is isomorphic to $\widetilde{W_{p,r-2}} \cong V^{\HVir}(h_{p,r}+p,h_I)$. Arrows represent $\Phi_{p} (\widetilde L, c)$ (descending), and $\widetilde {L(0)}$ (horizontal).} 
\label{def-diagram}
\end{figure}

\begin{figure}[H]
\begin{tikzpicture}[xscale=1.2,yscale=1.5]

\draw[thick] (0,0) --(1,1) --(2,0);
\draw[thick] (2.5,0) --(3.5,1) --(4.5,0);
\draw[thick] (5,0) --(6,1) --(7,0);
\node at (-1,0.5) {$\cdots$};
\node at (8,0.5) {$\cdots$};

\node[label=above:$v_{0,r}$,tokens=1] (b) at (3.5,1) {};
\node[label=above:$v_{0,r-2}$,tokens=1] (c) at (6,1) {};
\tikzstyle{token}=                 [fill=white,draw,circle,
                                    inner sep=0.5pt,minimum size=1ex]
\node[label=above:$v_{0,r+2}$,tokens=1] (a) at (1,1) {};

\path[->,dashed] (a) edge node [above] {$\widetilde{L(0)}$} (b)
			 (b) edge node [above] {$\widetilde{L(0)}$} (c)
			 (c) edge node [above] {$\widetilde{L(0)}$} (8,1)
			 (-1,1) edge node [above] {$\widetilde{L(0)}$} (a);
			 
\fill[pattern=dots] (2.5,0) to (3.5,1) to (4.5,0);
\fill[pattern=dots] (5,0) to (6,1) to (7,0);

\end{tikzpicture}
\caption{$\overline{\Pi(0)}$--module $\widetilde{\Pi (0,r)}$. Dotted area represents a portion of a deformed $\HVir$--module $\widetilde{W_{0,r}}$.}  
\label{diag-0}
\end{figure}
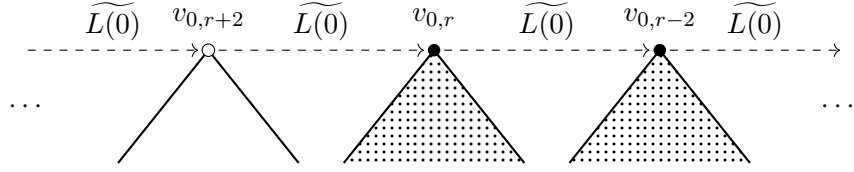

\begin{figure}[H]
\begin{tikzpicture}[xscale=0.8,yscale=0.98]

\draw[thick] (0,0) --(1,1) --(2,0);
\draw[thick] (2.5,0) --(5,2.5) --(7.5,0);
\draw[thick] (8,0) --(12,4) --(16,0);
\draw[dashed, thin] (4,0) --(5,1) --(6,0);
\draw[dashed, thin] (9.5,0) --(12,2.5) --(14.5,0);
\draw[dashed, thin] (11,0) --(12,1) --(13,0);

\node[label=above:$v_{-p,r+2}\quad$,tokens=1] (a) at (1,1) {};
\node[label=above:$v_{-p,r}$,tokens=1] (b) at (5,2.5) {};
\node[label=above:$\quad v_{-p,r-2}$,tokens=1] (c) at (12,4) {};
\node[tokens=1] (b1) at (5,1) {};
\node[label=below:$S_{p}(c)v_{-p,r}$,tokens=0] at (5.7,0.7) {};
\node[tokens=1] (c1) at (12,2.5) {};
\node[label=below:$S_{p}(c)v_{-p,r-2}$,tokens=0] at (12.8,2.1) {};
\node[tokens=1] (c2) at (12,1) {};
\node[label=below:$(S_{p}(c))^2v_{-p,r-2}$,tokens=0] at (13,0.7) {};
\node at (-0.25,0.75) {$\cdots$};
\node at (16.5,0.75) {$\cdots$};

\path[->,dashed] (a) edge [bend left=12.5] node [above] {$\widetilde{L(p)}$} (b)
			(b) edge [bend left=12.5] node [above] {$\widetilde{L(p)}$} (c)
			(b1) edge [bend left=12.5] (c1)
			(a) edge      node [above] {$\widetilde{L(0)}$}    (b1)
			(b) edge      node [above] {$\widetilde{L(0)}$}    (c1)
			(b1) edge     node [above] {$\widetilde{L(0)}$}    (c2);
\end{tikzpicture}
\caption{Deformed action of $\HVir$ on $\widetilde{\Pi(-p,r)}$, $p\in \N$}
\label{diag-1}
\end{figure}

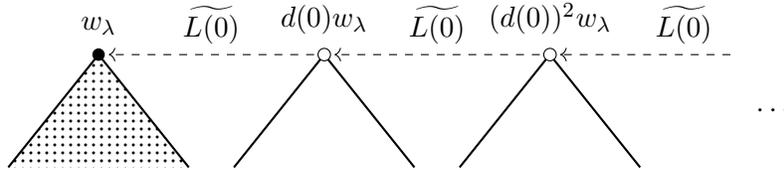
\begin{figure}[H]
\begin{tikzpicture}[xscale=1.2,yscale=1.5]

\draw[thick] (0,0) --(1,1) --(2,0);
\draw[thick] (2.5,0) --(3.5,1) --(4.5,0);
\draw[thick] (5,0) --(6,1) --(7,0);
\node at (8.5,0.5) {$\cdots$};

\node[label=above:$w_{\lambda}$,tokens=1] (a) at (1,1) {};
\tikzstyle{token}=                 [fill=white,draw,circle,
                                    inner sep=0.5pt,minimum size=1ex]
\node[label=above:$d(0)w_{\lambda}$,tokens=1] (b) at (3.5,1) {};
\node[label=above:$(d(0))^2w_{\lambda}$,tokens=1] (c) at (6,1) {};

\path[->,dashed] (b) edge node [above] {$\widetilde{L(0)}$} (a)
			 (c) edge node [above] {$\widetilde{L(0)}$} (b)
			 (8,1) edge node [above right] {$\widetilde{L(0)}$} (c);

\fill[pattern=dots] (0,0) to (1,1) to (2,0);

\end{tikzpicture}
\caption{A deformed Whittaker module $\widetilde{\Pi}_{\lambda} $. Dotted area represents $\widetilde{\Pi}_{\lambda}^{(0)} $ which is isomorphic to a self-dual $\HVir$--module $L^{\HVir}(\frac{c_L-2}{24}+\lambda,c_{L,I})$.} %highest weight $\HVir$--module.}
\label{diag-lambda}
\end{figure}


\begin{thebibliography}{999}

 \bibitem{ACKP} E. Arbarello, C. De Concini, V. Kac, C. Procesi, \emph{Moduli spaces of curves and representation theory}, Commun.\ Math.\ Phys.  117 (1988), 1--36.

\bibitem{AR1}D. Adamovi\'{c},  G. Radobolja, \emph{Free field realization of the twisted Heisenberg--Virasoro algebra at level zero and its applications}, Journal of Pure and Applied Algebra 219 (10) 2015, pp.\ 4322--4342

\bibitem{AR2}D. Adamovi\'{c},  G. Radobolja, \emph{On free field realization of $W(2,2)$--modules, SIGMA 12 (2016)}, 113, 13 pages,  arXiv:1605.08608

\bibitem{AM-Selecta}  D. Adamovi\' c, A. Milas, \emph{ Lattice construction of logarithmic modules for certain vertex algebras}, Selecta Mathematica, New Series 15 (2009) 535--561  

\bibitem{AdM-2010}  D. Adamovi\'c, A. Milas, \emph{On W-Algebras Associated to (2, p) Minimal Models and Their Representations},  International Mathematics Research Notices 2010 (2010) 20 : 3896--3934

\bibitem{AdM-2013}  D. Adamovi\'{c}, A. Milas, \emph{Vertex operator superalgebras and LCFT}, Journal of Physics A: Mathematical and Theoretical. 46 (2013) , 49; 494005, Special Issue on Logarithmic conformal field theory  

\bibitem{ALZ} D. Adamovi\'{c} R. Lu, K. Zhao, \emph{Whittaker modules for the affine Lie algebra $A_1 ^{(1)}$}, Advances in Mathematics 289 (2016) 438--479, arXiv:1409.5354

\bibitem{Bakalov} B. Bakalov, \emph{Twisted logarithmic modules of vertex algebras}, Commun.\ Math.\ Phys.\ 345 (2016), no.\ 1, 355-383.

\bibitem{BG} A. Bagchi, R. Gopakumar, \emph{Galilean conformal algebras and AdS/CFT}, JHEP 07 (2009) 037, arXiv:0902.1385 

\bibitem{BGMM} A. Bagchi, R. Gopakumar, I. Mandal, A. Miwa, GCA in 2D, JHEP 08 (2010) 004, arXiv:0912.1090

\bibitem{BJMN} N. Banerjee, D.P. Jatkar, S. Mukhi, T. Neogi, \emph{Free-field realisations of the $BMS_3$ algebra and its extensions}, JHEP 06 (2016) 024, arXiv:1512.06240

\bibitem{Billig}Y. Billig, \emph{Representations of the twisted Heisenberg--Virasoro algebra at level zero}, Canadian Math.\ Bulletin, 46 (2003), 529-537, arXiv:0201314v1

\bibitem{BDT} S. Berman, C. Dong and S. Tan, \emph{Representations of a class of lattice type vertex algebras}, J. of Pure and Applied Algebra 176 (2002) 27-47

\bibitem{CW} H. Guo and Q. Wang, \emph{Twisted Heisenberg-Virasoro vertex operator algebra}, arXiv:1612.06991.

\bibitem{CG} T. Creutzig, T. Gannon, \emph{Logarithmic conformal field theory, log-modular
tensor categories and modular forms}, arXiv:1605.04630

\bibitem{efren} E. Frenkel, \emph{Lectures on Wakimoto modules,
opers and the center at the critical level},  Adv.  Math 195 (2005)
297-404.

\bibitem{Huang} Y.-Z. Huang, \emph{Generalized twisted modules associated to general automorphisms of a vertex operator algebra}, Commun.\ Math.\ Phys.\ 298 (2010), 265-292.

\bibitem{HLL} Y.-Z. Huang, J. Lepowsky, L. Zhang, \emph{Logarithmic tensor category theory for generalized modules for a conformal vertex algebra}, Parts I-VIII, arXiv:1012.4193, 1012.4196, 1012.4197, 1012.4198, 1012.4199, 1012.4202, 1110.1929, 1110.1931; Part I published in Conformal Field Theories and Tensor Categories, Springer, Berlin-Heidelberg, 2014. 169--248.

\bibitem{JJ}  Q. Jiang, C. Jiang, \emph{Representations of the twisted Heisenberg-Virasoro algebra and the full toroidal Lie algebras}, Algebra Colloq., 2007, 14: 117--134.

\bibitem{LL}J. Lepowsky and H. Li, \emph{Introduction to Vertex Operator Algebras and Their Representations}, Birkh\"{a}user, Boston, 2003.

\bibitem{MSV} F. Malikov, V. Schechtman, A. Vaintrob, Chiral de Rham complex,  Comm. Math. Phys. 204 (1999) 439--473

\bibitem{Mil} A. Milas, \emph{ Weak modules and logarithmic intertwining operators for vertex operator algebras}, in Recent Developments in Infinite-Dimensional Lie Algebras and Conformal Field Theory, ed.\ S. Berman, P. Fendley, Y.-Z. Huang, K. Misra, and B.
Parshall, Contemp. Math., Vol. 297, American Mathematical Society, Providence,
RI, 2002, 201--225.

\bibitem{Miy} M. Miyamoto, \emph{Modular invariance of vertex operator algebra satisfying $C_2$--cofiniteness}, Duke Math.\ J. 122 (2004) 51-91.

\bibitem{R-2013} G. Radobolja, \emph{Subsingular vectors in Verma modules, and tensor product of weight modules over the twisted Heisenberg-Virasoro algebra and $W(2, 2) $ algebra}, Journal of Mathematical Physics 54 071701 (2013)

\bibitem{RR} J. Rasmussen and C. Raymond, \emph{Galilean contractions of $W$--algebras}, Nuclear Physics B, 922 435--479, 	arXiv:1701.04437 

\bibitem{TK} A. Tsuchiya and Y. Kanie, \emph{ Fock space representations of the Virasoro algebra -- Intertwining operators} Publ. Res. Inst. Math. Sci, 22:259--
327, 1986.


\end{thebibliography}
\end{document}